\documentclass[11pt]{amsart}
\usepackage{latexsym}
\usepackage{mathtools}
\usepackage{amsmath}
\usepackage{amssymb}
\usepackage{mathrsfs}
\usepackage[abbrev]{amsrefs}
\usepackage{graphicx}
\usepackage[all]{xy}
\usepackage{color}
\usepackage{enumerate}
\usepackage[normalem]{ulem}
\usepackage[utf8]{inputenc}
\usepackage{ifthen}
\usepackage{tikz}

\newcommand{\Env}[2][]{%
\ifthenelse{ \equal{#1}{} }  
{\ensuremath{#2_{\mathsf{c}}}}  
{\ensuremath{#2_{\mathsf{c},#1}}}
}

\newcommand{\LipOp}{\operatorname{\rm{Lip}}}

\newcommand{\Id}{\operatorname{\rm{Id}}}

\newcommand{\Lip}{\ensuremath{\mathrm{Lip}}}

 \newcommand{\Rea}{\mathbb{R}}
 \newcommand{\Nat}{\mathbb{N}}

\newcommand{\Compl}{\mathbb{C}}

\newcommand{\C}{\ensuremath{\mathcal{C}}}

\newcommand{\PP}{\ensuremath{\mathcal{P}}}
\newcommand{\MM}{\ensuremath{\mathcal{M}}}

\newcommand{\A}{\ensuremath{\mathcal{A}}}

\newcommand{\F}{\ensuremath{\mathcal{F}}}
\newcommand{\T}{\ensuremath{\mathcal{T}}}

\newcommand{\cospn}{\operatorname{\overline{\mathrm{span}}}\nolimits}

\newcommand{\lip}{\ensuremath{\mathrm{lip}}}
\newcommand{\complemented}{\stackrel{C}{\hookrightarrow}}

\allowdisplaybreaks

\newtheorem{Theorem}{Theorem}[section]
\newtheorem{thmx}{Theorem}

\newtheorem{Lemma}[Theorem]{Lemma}

\newtheorem{Corollary}[Theorem]{Corollary}

\theoremstyle{remark}
\newtheorem{Remark}[Theorem]{Remark}
\newtheorem{Definition}[Theorem]{Definition}
\newtheorem{Example}[Theorem]{Example}
\newtheorem{Question}[Theorem]{Question}
\newtheorem{Suggestion}[Theorem]{Suggestion}

\begin{document}
\title[Projections in Lipschitz-free spaces]{Projections in Lipschitz-free spaces induced by group actions}

\author[M. C\' uth]{Marek C\'uth}
\author[M. Doucha]{Michal Doucha}
\email{cuth@karlin.mff.cuni.cz}
\email{doucha@math.cas.cz}

\address[M.~C\' uth]{Charles University, Faculty of Mathematics and Physics, Department of Mathematical Analysis, Sokolovsk\'a 83, 186 75 Prague 8, Czech Republic}
\address[M.~C\'uth, M.~Doucha]{Institute of Mathematics of the Czech Academy of Sciences, \v{Z}itn\'a 25, 115 67 Prague 1, Czech Republic}


\thanks{M.~C\'uth has been supported by Charles University Research program No. UNCE/SCI/023, by the GA\v{C}R project 19-05271Y and by RVO: 67985840. M. Doucha was supported by the GA\v{C}R project 19-05271Y and RVO: 67985840.
}

\begin{abstract}
We show that given a compact group $G$ acting continuously on a metric space $\MM$ by bi-Lipschitz bijections with uniformly bounded norms, the Lipschitz-free space over the space of orbits $\MM/G$ (endowed with Hausdorff distance) is complemented in the Lipschitz-free space over $\MM$. We also investigate the more general case when $G$ is amenable, locally compact or SIN and its action has bounded orbits. Then we get that the space of Lipschitz functions $\Lip_0(\MM/G)$ is complemented in $\Lip_0(\MM)$. Moreover, if the Lipschitz-free space over $\MM$, $\F(\MM)$, is complemented in its bidual, several sufficient conditions on when $\F(\MM/G)$ is complemented in $\F(\MM)$ are given. Some applications are discussed. The paper contains preliminaries on projections induced by actions of amenable groups on general Banach spaces.
\end{abstract}
\maketitle

\section{Introduction}
Lipschitz-free Banach spaces are free objects in the category of Banach spaces over (pointed) metric spaces and also canonical preduals of spaces of scalar-valued Lipschitz functions on metric spaces. They currently form one of the most studied classes of spaces in Banach space theory. As their investigation is proceeding further, it is becoming apparent that the linear geometric structure of Lipschitz-free spaces is highly sensitive to geometric properties of metric spaces over which they are defined. Indeed, more and more results show the subtle relations between various standard properties studied in metric geometry and the structure of the corresponding Lipschitz-free spaces, see e.g. \cites{AlNew, APP21, GPZ18, PZ18}. In this note, we contribute to this research programme by studying Lipschitz-free spaces over spaces of orbits of compact (or more generally amenable) group actions by isometries (or more generally by bi-Lipschitz bijections with uniform bound on Lipschitz norms) equipped with the Hausdorff distance. Spaces of this form have a prominent role in metric geometry as many homogeneous metric space can be considered in this way. The following summarizes some of our main results.
\begin{thmx}\label{thm:intro}
Let $(\MM,d,0)$ be a pointed metric space and let $G$ be a group acting by isometries on $\MM$ such that some (or equivalently every) orbit is bounded. Denote by $\MM/G$ the space of closures of the orbits $\{\overline{Gx}\colon x\in\MM\}$ endowed with the Hausdorff metric.
\begin{itemize}
    \item If $G$ is compact, then $\F(\MM/G)\complemented \F(\MM)$.
    \item If $G$ is commutative or locally compact and amenable, then \[\Lip_0(\MM/G)\complemented \Lip_0(\MM).\] If $\MM$ is moreover proper and both $\MM$ and $\MM/G$ are purely $1$-unrectifiable, then $\F(\MM/G)\complemented \F(\MM)$.
\end{itemize}
\end{thmx}
One of the main application has already appeared in a separate paper in \cite{DP20} where it was used to prove the metric approximation property in free spaces over all compact metric groups, however other suggestions for applications are presented at the end of the paper. Nevertheless, the results seem to be of independent interest and after completing the first version of the paper, we found interesting that some of the questions have been also considered for Wasserstein spaces of measures by Lott and Villani in \cite{LoVi}*{Subsection 5.5} (see also \cite{GGKMS} for further generalizations).\medskip

Finally, let us say that the paper also has a secondary goal. That is, to put our new results about Lipschitz-free spaces into the context of more general actions of amenable groups on Banach spaces and the associated projections. We believe that these more general results are essentially known (and likely studied in the more general framework of amenable Banach algebras). However, since we lack any suitable reference we prove these results here which, we are sure, will be of benefit for most of the readers, who we do not want to expect to be familiar with amenable group actions.\medskip

Besides this introduction, the next section consists of preliminaries on Lipschitz-free spaces, amenable groups and basic notation concerning group actions on Banach spaces. In Section~\ref{sec:banach} we concentrate the aforementioned review of amenable group actions on Banach spaces. Then we present our main new results and the article is finished by suggestions for further applications and with some open problems.

Let us mention that we consider real Banach spaces, but many of the results have straightforward generalization to the complex case. Also, given a Banach space $X$, we denote by $\kappa_X:X\to X^{**}$ the canonical isometric embedding and for $A\subset X$ we put $A^\perp = \bigcap_{a\in A}\{x^*\in X^*\colon x^*(a)=0\}$, for $A\subset X^{*}$ we put $A_\perp = \bigcap_{a\in A}\{x\in X\colon a(x)=0\}$. If $X$ and $Y$ are Banach spaces, the symbol $Y\complemented X$ means that $Y$ is linearly isomorphic to a complemented subspace of $X$.

\section{Preliminaries}

The aim of this section is to recall the relevant information concerning Lipschitz-free spaces and amenable groups and settle a basic notation concerning groups actions on Banach spaces.
\subsection{Lipschitz-free spaces}

Lipschitz-free spaces, also known as Arens-Eells spaces, are free objects in the category of Banach spaces over the category of (pointed) metric spaces and they are characterized by the following universal property.

\begin{Theorem}\label{thm:universalProperty}Let $(\MM,\rho)$ be a pointed metric  space. Then $\F(\MM)$ is the unique Banach space for which there exists a linear isometry $\delta:\MM\to \F(\MM)$ onto a linearly dense set such that for every Banach space $X$ and every Lipschitz map $f\colon\MM\rightarrow X$ with $f(0) = 0$ there
exists a unique  linear map $T_{f}\colon\F_{p}(\MM)\rightarrow X$ with $T_{f}  \circ \delta_{\MM} = f$. Moreover $\Vert T_{f} \Vert =\LipOp(f)$. Pictorially,
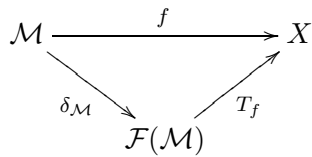
\[
\xymatrix{\MM \ar[rr]^f \ar[dr]_{\delta_{\MM}} & & X\\
& \F(\MM) \ar[ur]_{T_{f}} &}
\]
\end{Theorem}
It follows (considering the case when $X=\Rea$ in Theorem ~\ref{thm:universalProperty}) that $\F(\MM)$ is canonical preduals of the Banach space of scalar-valued Lipschitz functions vanishing at a distinguished point, and this point of view also allows the fastest construction of the space $\F(\MM)$. That is, let $\MM$ be a metric space with a distinguished point $0$ (different choices will produce linearly isometric spaces though) and let $\Lip_0(\MM)$ denote the Banach space of all scalar-valued Lipschitz function vanishing at $0$, with the norm being the minimal Lipschitz constant. The \emph{Lipschitz-free space over $\MM$}, denoted by $\F(\MM)$, can be defined as the closed linear span of $\{\delta(m)\colon m\in\MM\}\subseteq \Lip_0(\MM)^*$ in the dual of $\Lip_0(\MM)$, where $\delta(m)$ is the evaluation functional, that is, $\delta(m)(f) = f(m)$ for every $f\in\Lip_0(\MM)$. It is not so difficult to verify that this construction gives a Banach space $\F(\MM)$ which satisfies Theorem~\ref{thm:universalProperty}, see e.g. \cite{CDW2016}*{Section 2} for a proof, some more references and basic facts. Let us emphasize that the $w^*$ topology on bounded subsets of $\Lip_0(\MM) \equiv (\F(\MM))^*$ then coincides with pointwise convergence (which easily follows using the fact that $\delta(\MM)$ is linearly dense in $\F(\MM)$).

There is also an abstract construction of $\F(\MM)$ which avoids using the space $\Lip_0(\MM)$ and which defines the norm on $\F(\MM)$ as a variant of the Kantorovich-Rubinstein distance and connects the theory with optimal transport. We refer to \cite{Weaver1999} for a proper introduction to the subject.

\subsection{Amenable groups and related notions}
Let $G$ be a topological Hausdorff group. Then $G$ acts on $\ell_\infty(G)$ by left-shift and by right-shift, that is, for any $f\in \ell_\infty(G)$ and $g,h\in G$ we have $g\cdot f(h)=f(g^{-1}h)$ and $f\cdot g(h)=f(hg^{-1})$, respectively.

 By $\C^b_{lu}(G)$ and $\C^b_{ru}(G)$ we denote the closed subspaces of $\ell_\infty(G)$ defined as
\[\begin{split}\C^b_{lu}(G):=\{f\in\ell_\infty(G)\colon & x\mapsto x^{-1}\cdot f\in \ell_\infty(G) \text{ is continuous map }\\ & \text{ from $G$ to $(\ell_\infty(G),\|\cdot\|_\infty)$}\},
\end{split}\]
\[\begin{split}\C^b_{ru}(G):=\{f\in\ell_\infty(G)\colon & x\mapsto f\cdot x^{-1}\in \ell_\infty(G) \text{ is continuous map }\\ & \text{ from $G$ to $(\ell_\infty(G),\|\cdot\|_\infty)$}\},
\end{split}\]
 Let us note that $\C^b_{lu}(G)$ and $\C^b_{ru}(G)$ are exactly the spaces of all bounded left-uniformly continuous and right-uniformly continuous functions from $G$ to the scalar field, respectively. We denote by $\C^b_{u}(G)$ the set $\C^b_{lu}(G)\cap \C^b_{ru}(G)$.

We say that a linear functional $M:\C^b_{lu}(G)\to \Compl$ is \emph{left-invariant mean on $G$} if the following conditions are satisfied:
\begin{enumerate}
    \item\label{it:one} $M(1_G) = 1$;
    \item\label{it:invariance} $M$ is left-invariant, that is, $M(f)=M(g\cdot f)$ for all $f\in \C^b_{lu}(G)$ and $g\in G$;
    \item\label{it:positive} $M(f)\geq 0$ for all $f\in \C^b_{lu}(G)$ with $f\geq 0$.
\end{enumerate}
Similarly, $M:\C^b_{ru}(G)\to \Compl$ is \emph{right-invariant mean on $G$} if we replace left-invariance of $M$ by right-invariance of $M$.
Recall that for any left-invariant of right-invariant mean $M$ we have $\|M\|\leq 1$ (see e.g. \cite{AmenableBook}*{p. 422}).

\begin{Definition}
$G$ is said to be \emph{amenable} if there exists a left-invariant mean $M:\C^b_{lu}(G)\to \Compl$. 
\end{Definition}

\begin{Remark}
Note that the existence of left-invariant and right-invariant means are equivalent since once $M$ is left-invariant mean, then $\widetilde{M}$ defined as $\widetilde{M}(f):=M(g\mapsto f(g^{-1}))$ is right-invariant mean and vice versa.

Once we have a left-invariant mean $M_l$ on $G$, it is easy to see that the mapping $M:\C^b_{u}(G)\to\Compl$ defined as $M(f):=M_l(g\mapsto M_l(fg))$, $f\in \C^b_{u}(G)$ is a bi-invariant mean, that is, $M(1)=1$, $M(f)\geq 0$ for $f\in \C^b_{u}(G)$ with $f\geq 0$ and $M(f\cdot g)=M(g\cdot f) = M(f)$ for $f\in \C^b_{u}(G)$ and $g\in G$.
\end{Remark}

In the special case when $G$ is locally compact, it is more convenient to consider invariant means on larger function algebras than $C^b_{lu}(G)$. Recall that every locally compact group $G$ comes equipped with a left-invariant Haar measure $\mu$ which is unique up to multiplication by a positive scalar. Moreover, if $G$ is compact, then $\mu$ is finite, so without loss of generality it is probabilistic, and moreover it is bi-invariant, i.e. both left and right invariant. We can then consider $L_\infty(G,\mu)\supseteq C^b_{lu}(G)$, further denoted just by $L_\infty(G)$, on which $G$ acts like on $C^b_{lu}(G)$ by the shift. Amenable locally compact groups admit a left-invariant mean (or more generally, a bi-invariant mean) on $L_\infty(G)$ with the same properties as on $C^b_{lu}(G)$, we refer to \cite{AmenableBook}*{Appendix G} for details. If $G$ is compact, then the bi-invariant mean can be given by the formula 
\[M(f):=\int_G f(g) d\mu(g),\quad f\in L_\infty(G),\]
where $\mu$ is the bi-invariant probabilistic Haar measure on $G$, and it is moreover $w^*$-continuous.

\noindent\textbf{Convention.} Therefore, when $G$ is locally compact and amenable, the bi-invariant mean associated to $G$ is always assumed to be defined on $L_\infty(G)$.

The basic examples of amenable groups are abelian and compact groups, we again refer to \cite{AmenableBook}*{Appendix G} for more.\medskip

We will also need the notion of a SIN group. Well-known examples are e.g. compact groups, discrete groups or commutative groups, see e.g. \cite{SINgroupBook} for more details.

\begin{Definition}
Let $G$ be a group with the identity element $e$. We say $G$ is SIN (\emph{small invariant neighborhoods}) if for every neighborhood $U$ of $e$ there exists a neighborhood $V$ of $e$ such that $V = gVg^{-1}\subset U$ for every $g\in G$.
\end{Definition}

\subsection{Actions of groups on Banach spaces}\label{subsec:notationActions}

Suppose that a group $G$ acts by isometries on a Banach space $X$. If $X$ is real, which is always going to be our case, such isometries are automatically affine by the Mazur-Ulam theorem. We shall however always write \emph{affine isometries} when we want to make clear that they are not necessarily linear. Notice also that when $G$ acts by affine isometries, then for any $g\in G$ the map $x\mapsto gx-g0$ is a linear surjective isometry. 
Given an affine action $\A:G\times X\to X$ by isometries we denote by $\pi(\A)$ the associated linear part of the action, that is, we have $\pi(\A)(g,x) = \A(g,x) - \A(g,0)$ for every $g\in G$ and $x\in X$. We shall often not use any special symbol for the given affine action $G\curvearrowright X$ and write simply $gx$ for the result of acting by $g\in G$ on $x\in X$ and similarly we often write $\pi(g)x$ for $gx-g0$. 

Note that in this case $(\pi(g))_{g\in G}$ defines an action of $G$ on $X$ by linear isometries. Indeed, for every $x\in X$ and $g,h\in G$ we have
\[\begin{split}\pi(g)(\pi(h)(x)) & = \pi(g)(hx-h0) = \pi(g)(hx)-\pi(g)(h0)\\ & = (ghx-g0)-(gh0-g0) = ghx-gh0=\pi(gh)(x).\end{split}\]
and $\pi(e)x = ex-e0 = x-0=x$.

Notice that when a group $G$ acts by linear isometries on a Banach space $X$, then there is a dual action by linear isometries of $G$ on $X^*$ defined by $gf(x):=f(g^{-1}x)$, where $g\in G$, $f\in X^*$, and $x\in X$.

If $G$ acts on $X$ by affine isometries, then by $\pi(g)x^*$ we denote dual action of $\pi(g)$ on $X^*$, that is, $\pi(g)x^*(x) = x^*(\pi(g^{-1})x)$, $x\in X$.

\section{Projections associated to group actions in a general Banach space}\label{sec:banach}

The starting point here is an observation that there is a very general method of constructing a linear projection from an action of a compact group on a Banach space (see part \eqref{it:cpctGroup} of Lemma~\ref{lem:amenableOperator}). The main outcome of this section is contained in Lemma~\ref{lem:amenableOperator} and Theorem~\ref{thm:amenableLinearOperator}, where we summarize what analogies are available in the more general setting of amenable groups. 

The results from this section will be used in Section~\ref{sec:free}, where everything is investigated in the setting of Lipschitz-free spaces.

We start with a lemma which will be used several times in what follows. Note that if $G$ acts on a Banach space $X$ by linear isometries, then its orbits $Gx$ are bounded in diameter by $2\|x\|$ for every $x\in X$, so the following result applies in particular to actions by linear isometries.

\begin{Lemma}\label{lem:mostGeneralAction}
Let $X$ be a Banach space, $G$ be a group acting continuously on $X$ by isometries such that some (or equivalently every) orbit is bounded. Pick $x\in X$, $x^*\in X^*$ and let $\varphi:G\to\Rea$ be defined by $\varphi(g)=x^*(gx)$, $g\in G$. Then $\varphi\in \C^b_{ru}(G)$. Moreover,
\begin{enumerate}
    \item if $G$ is locally compact, then $\varphi\in L_\infty(G)$; and
    \item if $G$ is SIN, then $\varphi\in \C^b_{u}(G)$.
\end{enumerate}
\end{Lemma}
\begin{proof}
Since the orbit $Gx$ is bounded, we obtain $\varphi\in\ell_\infty(G)$.

In order to check that $\varphi\in\C^b_{ru}(G)$, pick $\varepsilon > 0$ and using the continuity of the map $h\mapsto hx$ pick a neighborhood $U$ of $e$ with $\|hx-x\|<\varepsilon$ for every $h\in U$. Then for every $h\in U$ we obtain
\[\begin{split}\|\varphi\cdot h^{-1}   - \varphi\|_{\infty} & = \sup_{g\in G}|x^*(ghx)-x^*(gx)| \\ & \leq \|x^*\|\sup_{g\in G}\|ghx-gx\|  = \|x^*\|\|hx-x\|\\ &  <  \varepsilon\|x^*\|,\end{split}\]
so the mapping $h\mapsto  \varphi\cdot h^{-1}$ is continuous and we have $\varphi\in \C^b_{ru}(G)$.

Assume that $G$ is SIN. In order to check that $\varphi\in\C^b_{lu}(G)$, pick $\varepsilon > 0$ and using the continuity of the map $h\mapsto hx$ pick a neighborhood $U$ of $e$ with $\|hx-x\|<\varepsilon$ for every $h\in U$. Since $G$ is SIN, we find neighborhood $V$ of $e$ with $gVg^{-1}\subset U$ for every $g\in G$. Then for every $h\in V$ we obtain
\[\begin{split}\|h^{-1}\cdot \varphi  - \varphi\|_{\infty} & =  \sup_{g\in G}|x^*(hgx)-x^*(gx)| \\ & \leq \|x^*\|\sup_{g\in G}\|hgx-gx\|\\ &  = \|x^*\|\sup_{g\in G}\|g^{-1}hgx-x\| < \varepsilon\|x\|\end{split},\]
which shows that $\varphi\in \C^b_{lu}(G)$.

Finally, if $G$ is locally compact, then it suffices to observe that $\varphi$ is continuous and since we already know it is bounded, we obtain $\varphi\in L_\infty(G)$.
\end{proof}

\begin{Lemma}\label{lem:amenableOperator}Let $X$ be a Banach space, $G$ an amenable group acting continuously by affine isometries on $X$ such that $G$ is either locally compact or SIN. Let $M$ be bi-invariant mean on $G$. Then the mapping $R_G:X\to X^{**}$ defined as
\[R_G(x)(x^*):=M\big(g\mapsto x^*(\pi(g)\cdot x)\big), \quad x\in X, x^*\in X^*\]
is well-defined linear operator with $\|R_G\|\leq 1$ and the following holds.
\begin{enumerate}
    \item\label{it:invariantRg} For every $g\in G$ and $x\in X$ we have \[\pi(g)\cdot R_G(x) = R_G(x) = R_G(\pi(g)x).\]
    \item\label{it:rangeRg} 
    \[R_G[X]\cap X=\{x\in X\colon \pi(g)x = x\text{ for every }g\in G\}=:I_G\]
    and $R_G(x)=x$ for every $x\in I_G$.
    \item\label{it:projection} $\ker R_G\cap I_G = \{0\}$, and the canonical projection of $(\ker R_G) + I_G$ onto $I_G$ along $\ker R_G$ has norm at most $1$.
    
    \item\label{it:dualProjection} The mapping $\widetilde{P_G}:=(R_G)^*\circ \kappa_{X^*}:X^*\to X^*$ is a norm-one projection with
    \[\widetilde{P_G}[X^*] = \{x^*\in X^*\colon \pi(g)x^* = x^*\}=:\widetilde{I_G}.\]
    \item\label{it:kernel} We have
    \[\ker R_G = (\widetilde{I_G})_\perp.\]
\end{enumerate}

    Moreover,
    \begin{enumerate}[\upshape (a)]
        \item if some (or equivalently every) orbit is bounded, then for the operator $R_G$ we have $R_G(gx)=R_G(x)$ and $\widetilde{P_G}(x^*)(gx)=\widetilde{P_G}(x^*)(x)$ for every $x\in X$, $x^*\in X^*$ and $g\in G$;
        \item\label{it:cpctGroup} if the group $G$ is compact, then $R_G:X\to X$ is a norm-one projection onto $I_G$ and
    \begin{equation}\label{eq:cpctProjection}R_G(x) = \int_G \pi(g)x \; d\mu(g),\quad x\in X,\end{equation}
    where $\mu$ is the bi-invariant probability Haar measure on $G$.
    \end{enumerate}
\end{Lemma}
\begin{proof}It follows directly from the corresponding definitions and from Lemma~\ref{lem:mostGeneralAction} (applied to the action $(\pi(g))_{g\in G}$ on $X$ by linear isometries) that $R_G$ is well-defined, linear and $\|R_G\|\leq 1$. Pick $g\in G$, $x^*\in X^{*}$ and $x\in X$. Using the left-invariance of $M$ we have
\[\begin{split}
(\pi(g)\cdot R_G(x)) (x^*) & = R_G(x)(\pi(g^{-1})\cdot x^*) = M\Big(h\mapsto x^*(\pi(gh)x)\Big)\\ &  = M\Big(h\mapsto  x^*(\pi(h)\cdot x) \Big) =  R_G(x)(x^*),
 \end{split}\]
 so we have $\pi(g)\cdot R_G(x) = R_G(x)$ and $R_G[X]\cap X\subset I_G$. Similarly, right-invariance of $M$ gives $R_G(\pi(g)x) = R_G(x)$, which implies \eqref{it:invariantRg}.
 
 Pick $x\in I_G$, then for every $x^*\in X^*$ we have
 \[R_G(x)(x^*) = M\Big(g\mapsto x^*(\pi(g)x)\Big) = M\Big(g\mapsto x^*(x)\Big) = x^*(x),\]
 which implies that $I_G\subset R_G[X]$ and $R_G(x)=x$ for every $x\in I_G$, so \eqref{it:rangeRg} holds.
 
For \eqref{it:projection} we observe that for any $x\in \ker R_G$ and $y\in I_G$ we obtain $\|y\| = \|R_G(x+y)\|\leq \|R_G\|\|x+y\|$ from which \eqref{it:projection} follows.

For \eqref{it:dualProjection} we observe that for $x\in X$ and $x^*\in X^*$ we have
\[\widetilde{P_G}(x^*)(x) = R_G^*\kappa_{X^*}(x^*)(x) = \kappa_{X^*}(x^*)(R_G x) = R_G(x)(x^*),\]
which, using \eqref{it:invariantRg}, implies that for every $g\in G$
\[\begin{split}
    \pi(g)\widetilde{P_G}(x^*)(x) & = \widetilde{P_G}(x^*)(\pi(g^{-1}x)) = R_G(\pi(g^{-1})x)(x^*) = R_G(x)(x^*)\\ & = \widetilde{P_G}(x^*)(x),
\end{split}\]
so we have $\widetilde{P_G}[X^*]\subset \widetilde{I_G}$. Conversely, for $x^*\in \widetilde{I_G}$ we obtain
\[\begin{split}\widetilde{P_G}(x^*)(x) & = R_G(x)(x^*) = M\Big(g\mapsto x^*(\pi(g)x)\Big)\\ & = M\Big(g\mapsto \pi(g^{-1})x^*(x)\Big) = M\Big(g\mapsto x^*(x)\Big) = x^*(x),\end{split}\]
so we have $\widetilde{P_G}(x^*) = x^*$ and so $\widetilde{P_G}|_{\widetilde{I_G}} = \Id$, which implies \eqref{it:dualProjection}.

For \eqref{it:kernel} pick $x\in X$. If $x\in \ker R_G$, then for every $x^*\in \widetilde{I_G}$ we have
\[
x^*(x) = \widetilde{P_G}(x^*)(x) = R_G(x)(x^*) = 0.
\]
On the other hand, if $x\in (\widetilde{I_G})_\perp$, then for every $x^*\in X^*$ we obtain
\[
R_G(x)(x^*) = \widetilde{P_G}(x^*)(x) = 0,
\]
so we have $x\in \ker R_G$.

 For the ``Moreover'' part we compute (using Lemma~\ref{lem:amenableOperator} by which all the evaluations of the function $M$ below make sense)
 \[\begin{split}
    R_G(gx)(x^*) & = M\Big(h\mapsto x^*(\pi(h)\cdot gx)\Big)\\ & = M\Big(h\mapsto x^*(\pi(hg)\cdot x)\Big) + M\Big(h\mapsto x^*(\pi(h)\cdot g0)\Big)\\ & = R_G(x)(x^*) + M\Big(h\mapsto x^*(hg0)\Big) - M\Big(h\mapsto x^*(h0)\Big)\\ & = R_G(x)(x^*),
\end{split}\]
where in the last two equalities we used the right-invariance of $M$.

Finally, if $G$ is compact, using the $w^*$-continuity of $M$ we have that $R_G(x)\in (X^*,w^*)^* = X$ and using that in this case $M$ is given by $M(f) = \int_G f(g) \; d\mu(g)$, we obtain the formula \eqref{eq:cpctProjection}.
\end{proof}

As we have seen above, if $G$ is compact we obtain that $I_G$ is complemented in $X$. This may be achieved also for a more general class of groups $G$ under the assumption that $X$ is complemented in its bidual, see Theorem~\ref{thm:amenableLinearOperator}. As we shall see later, the assumptions of Theorem~\ref{thm:amenableLinearOperator} are satisfied e.g. if $X$ is a dual space or if it is an $L$-embedded space.

\begin{Theorem}\label{thm:amenableLinearOperator}
Let $X$ be a Banach space and $G$ amenable group acting continuously on $X$ by affine isometries such that $X$ is complemented in its bidual via a projection $Q$ satisfying $\pi(G)\cdot \ker Q\subset \ker Q$, that is, $\pi(g)\cdot x^{**}\in \ker Q$ for every $g\in G$ and $x^{**}\in \ker Q$. Assume $G$ is locally compact or SIN and $M$ is bi-invariant mean on $G$.

If $R_G:X\to X^{**}$ and $I_G$ are as in Lemma~\ref{lem:amenableOperator}, the mapping $P_G:=Q\circ R_G:X\to X$ is projection of norm $\leq \|Q\|$ with $P_G[X] = I_G$ and $P_G[X]$ is a complemented subspace of $R_G[X]$.

Moreover, if some (or equivalently every) orbit is bounded, then $P_G(x) = P_G(gx)$ for every $g\in G$ and $x\in X$.
\end{Theorem}
\begin{proof}By Lemma~\ref{lem:amenableOperator}, we have that $P_G$ is linear operator with $\|P_G\|\leq \|Q\|$ and $\pi(g)\cdot R_G(x) = R_G(x)$ for every $x\in X$ and $g\in G$.

We \emph{claim} that $\pi(g)\cdot P_G(x) = P_G(x)$ for every $x\in X$ and $g\in G$. Actually, something more general holds.
\begin{equation}\label{claim:invariant}
\forall\zeta\in X^{**} \; \forall h\in G:\qquad \pi(h)\zeta = \zeta\Rightarrow \pi(h) Q\zeta = Q\zeta.\end{equation}
Indeed, pick $\zeta\in X^{**}$ and $h\in G$ with $\pi(h)\zeta = \zeta$. Then we have
 \[\zeta = Q\zeta + (I-Q)(\zeta)\]
 and also
 \[\zeta = \pi(h) Q\zeta + \pi(h) (I-Q) (\zeta),\]
 so since we have $\pi(h)Q\zeta \in X$ and $\pi(h)(I-Q)(\zeta)\in \ker Q$ by the assumption, using the uniqueness of the decomposition we obtain that $\pi(h)Q\zeta = Q\zeta$, which proves \eqref{claim:invariant} and so the claim is proved.
 
  By the above, we have $P_G[X]\subset I_G$. Moreover, by Lemma~\ref{lem:amenableOperator} we have $R_G(x)=x$ for every $x\in I_G$ which together with $P_G[X]\subset I_G$ implies that $P_G$ is a projection onto $I_G$.
 
 By Lemma~\ref{lem:amenableOperator} we have $P_G[X]= R_G[X]\cap X$, so $Q|_{R_G[X]}$ is a projection onto $P_G[X]$.

For the proof of the ``Moreover'' part we use the ``Moreover'' part of Lemma~\ref{lem:amenableOperator}.
\end{proof}

Let us observe that the assumptions of the previous result are met for example if $X$ is a dual space or if it is an $L$-embedded space.

\begin{Lemma}\label{lem:dualSpacesWorse}
Let $Y$ be a Banach space and $G$ be a group acting by affine isometries on $Y^*$ such that $\pi(g)x\in Y$ for every $x\in Y$. Then $Q:=\kappa_{Y^{*}}\circ (\kappa_Y)^*:Y^{***}\to Y^{***}$ is norm-one projection with $Q[Y^{***}] = \kappa_{Y^*}(Y^*)$ satisfying $\pi(G)\cdot \ker Q\subset \ker Q$.

Moreover, if $G$ is locally compact or SIN and it acts continuously on $Y^*$, $M$ is bi-invariant mean on $G$ and $R_G$, $P_G$ are as in Theorem~\ref{thm:amenableLinearOperator} applied to the space $X = Y^*$, then for every $x\in X$ we have $P_G(x) = R_G(x)|_Y$.
\end{Lemma}
\begin{proof}
It is well-known (and easy to check) that $Q$ is norm-one projection onto $\kappa_{Y^*}(Y^*)$ with $\ker Q = \{f\in Y^{***}\colon f\circ \kappa_Y \equiv 0\}$. Pick $f\in\ker Q$, $g\in G$ and $y\in Y$. By the assumption there exists $z\in Y$ with $\pi(g^{-1})\kappa_Yy = \kappa_Yz$, so we have
\[\pi(g)f(\kappa_Yy) = f(\pi(g^{-1})\kappa_Yy) = f(\kappa_Yz) = 0,\]
so $\pi(g)f\in \ker Q$.

For the ``Moreover'' part we observe that for $x\in X$ and $y\in Y$ we have
\[P_G(x)(y) = (QR_G(x))(\kappa_Yy) = ((\kappa_Y)^*R_G(x))(y) = R_G(x)(\kappa_Yy).\]
\end{proof}

The following applies in particular e.g. to $L$-embedded spaces, or more generally to $L^p$-summands in the bidual for every $p\in [1,\infty)$.

\begin{Lemma}\label{lem:lembed}
Let $X$ be a Banach space $Q:X^{**}\to X^{**}$ a projection satisfying that $Q[X^{**}] = \kappa_{X}[X]$ and
\[
\|f\| > \|f-Qf\|, \quad f\notin \ker Q. 
\]
If $G$ is a group acting by affine isometries on $X$, then we have $\pi(G)\cdot \ker Q\subset \ker Q$.
\end{Lemma}
\begin{proof}
Pick $f\in \ker Q$ and $g\in G$. In order to get a contradiction, assume that $\pi(g)f\notin \ker Q$. Pick $x\in X\setminus\{0\}$ and $y'\in \ker Q$ such that $\pi(g)f = x + y'$. Then we have $\|\pi(g)f\| > \|y'\|$ and since $G$ acts by isometries, we obtain $\|f\| > \|y'\|$. Moreover, since the action $\pi(G)$ is linear, we obtain $f = \pi(g^{-1})x + \pi(g^{-1})y'$ and since $G$ acts on $X$ we have $\pi(g^{-1})x\in X$ and $\pi(g^{-1})y' = - \pi(g^{-1})x + f$, so $\pi(g^{-1})y'\notin\ker Q$ and $\pi(g^{-1})y' - Q\pi(g^{-1})y' = f$. Thus, we obtain
\[\|f\| > \|y'\| = \|\pi(g^{-1})y'\| > \|\pi(g^{-1})y' - Q\pi(g^{-1})y'\| = \|f\|, \]
a contradiction.
\end{proof}

\section{Projections associated to group actions in Lipschitz-free spaces}\label{sec:free}

In this main part of our paper we aim at identifying the images of the mappings $R_G$ and $P_G$ from Lemma~\ref{lem:amenableOperator}, and Theorem~\ref{thm:amenableLinearOperator} in the case the Banach space we deal with is a Lipschitz-free space. The main abstract outcomes of our considerations are concentrated in Subsection~\ref{subsec:mainSubsection} and some special cases having connection to the class of purely 1-unrectifiable metric spaces are explained in Subsection~\ref{subsec:mainSubsectionDual}. Theorem~\ref{thm:intro} follows from Theorem~\ref{thm:freeSpaceAmenableOperator}, its consequence Corollary~\ref{cor:mainCpctGroup} and from Theorem~\ref{thm:dual}.

In Subsection~\ref{subsec:prelim} we settle some basic notation and explain why it suffices to consider the case of actions by isometries. Then we describe the quotient space $\MM/G$ and identify both the Lipschitz-free space over $\MM/G$ and its dual. Finally, we come to our main abstract results and some special cases related to purely 1-unrectifiable metric spaces.

\subsection{Preliminaries}\label{subsec:prelim} Fix an amenable group $G$ and suppose that $G$ acts continuously by bi-Lipschitz isomorphisms on a metric space $\MM$ so that the Lipschitz constants of the maps are uniformly bounded. First we use the standard trick to reduce such a general case to the case when $G$ acts by isometries.
\begin{Lemma}\label{lem:averagingmetric}
Let $G$ be an amenable group which acts continuously on a metric space $(\MM,d)$ by bi-Lipschitz isomorphisms whose norms are uniformly bounded. Then there exists a new metric $D$ on $\MM$, bi-Lipschitz equivalent to $d$, such that the formally same action of $G$ on $(\MM,D)$ is by isometries.

Quantitatively, if there are $0<r<R<\infty$ with  \begin{equation}\label{eq:biLipMetric}
    r d(gx,gy)\leq d(x,y)\leq R d(gx,gy),\qquad x,y\in\MM,\; g\in G,
    \end{equation}then we have
$rD(x,y)\leq d(x,y)\leq RD(x,y)$ for every $x,y\in\MM$.
\end{Lemma}
\begin{proof}By assumption, there are $0<r<R<\infty$ such that \eqref{eq:biLipMetric} holds. Let $M$ be a right-invariant mean on $G$. For any $x,y\in \MM$ set $$D(x,y):=M(g\mapsto d(gx,gy)).$$

In order to see that this is a well-defined map, we need to see the mapping $\varphi(g):=d(gx,gy)$, $g\in G$ is in the domain of $M$. It is indeed a bounded function on $G$ as the action of $G$ is by uniformly bounded bi-Lipschitz maps. Pick $\varepsilon>0$. Using the continuity of the map $h\mapsto hx$ and $h\mapsto hy$ pick a neighborhood $U$ of $e$ with $d(hx,x) + d(hy,y)<r\varepsilon$ for every $h\in U$. Then for every $h\in U$ we have
\[\begin{split}\|\varphi \cdot h^{-1} - \varphi\|_\infty & = \sup_{g\in G} |d(ghx,ghy) - d(gx,gy)|\\ & \leq \sup_{g\in G} d(ghx,gx) + d(ghy,gy)\\ & \leq r^{-1} \sup_{g\in G} d(hx,x) + d(hy,y) < \varepsilon, \end{split}\]
which shows that $\varphi\in\C^b_{ru}(G)$ and $M(\varphi)$ is defined.

We check that $D$ is a metric on $\MM$ bi-Lipschitz to $d$ and that the action of $G$ on $\MM$ is by isometries with respect to $D$. Pick $x,y,z\in\MM$.
\begin{itemize}
    \item If $x\neq y$, then we have $\inf_{g\in G} d(gx,gy)\geq rd(x,y)$, thus we obtain $D(x,y)\geq rd(x,y)>0$.
    \item The symmetry $D(x,y)=D(y,x)$ is obvious.
    \item For every $g\in G$, $d(gx,gz)\leq d(gx,gy)+d(gy,gz)$. Thus for functions $f_1,f_2,f_3\in \C^b_{lu}(G)$ defined by $g\mapsto d(gx,gz)$, $g\mapsto d(gx,gy)$, and $g\mapsto d(gy,gz)$, respectively, we have $M(f_2)+M(f_3)=M(f_2+f_3)\geq M(f_1)$ by the positivity of $M$, so $D(x,z)\leq D(x,y)+D(y,z)$.
    \item Using the positivity of $M$, we have \[r D(x,y)\leq d(x,y)\leq RD(x,y)\] and so the metric $D$ is bi-Lipschitz equivalent to $d$.
    \item Finally, for every $x,y\in\MM$ and $g\in G$ we need to check that $D(x,y)=D(gx,gy)$ which however follows from the right-invariance of the mean.
\end{itemize}
\end{proof}
Thus, from now on, we will assume that $G$ acts by isometries on a metric space $\MM$. Then $G$ induces actions by affine isometries on both $\F(\MM)$ and $\Lip_0(\MM)$. Unless $G$ preserves $0$ in $\MM$, these actions are not linear.

More precisely, for every $g\in G$ the mapping $T'_g:\MM\rightarrow \F(\MM)$ defined by $m\mapsto \delta(gm)-\delta(g0)$ is an isometry preserving $0$, so it extends to a linear isometry $T_g:\F(\MM)\rightarrow \F(\MM)$. Since for every $g,h\in G$, $T_g\circ T_h=T_{gh}$, it follows that each $T_g$ is a linear surjective isometry. Denote also by $A_g:\F(\MM)\rightarrow\F(\MM)$ the affine map $x\to T_g(x)+\delta(g0)$. It is now easy to check that $A_g$ is an affine isometry extending the isometry $g:\MM\rightarrow\MM$. To conclude, the original action of $G$ on $\MM$ provided us with an action by affine isometries on $\F(\MM)$ by the maps $(A_g)_{g\in G}$ and with an action by linear isometries by the maps $(T_g)_{g\in G}$. We denote by $\A_G:G\times \F(\MM)\to \F(\MM)$ the affine action given by $\A(g,\mu) = A_g(\mu)$, $g\in G$, $\mu\in\F(\MM)$. Note that, using the notation from Subsection~\ref{subsec:notationActions}, we have $\pi(\A_G)(g,\mu) = T_g(\mu)$ for every $g\in G$ and $\mu\in\F(\MM)$. In what follows we denote the action $\pi(\A_G)$ by $\T_G$. It is straightforward to check that the action $\A_G$ is norm-continuous. Indeed, since the action $G$ on $\MM$ is continuous and $\delta(\MM)$ is linearly dense in $\F(\MM)$ we obtain that for every $\mu\in\F(\MM)$ the mapping $g\mapsto g\mu$ is continuous, so for every $g\in G$, $\mu\in\F(\MM)$ and $\varepsilon>0$ there exists an open neighborhood $U$ of $g$ such that $\|A_hx-A_gx\|<\tfrac{\varepsilon}{2}$ for every $h\in U$ which, using that the action is by isometries, implies that for every $(h,y)\in U\times B(x,\tfrac{\varepsilon}{2})$ we have 
\[\|A_hy-A_gx\|\leq \|A_hy-A_hx\| + \tfrac{\varepsilon}{2} < \varepsilon,\]
so the action $\A_G$ is continuous.

Notice that we may consider the dual linear action of $G$ on $\Lip_0(\MM)$ by considering $S_g(f)(m):=f(g^{-1}m) - f(g^{-1}0)$, for $f\in\Lip_0(\MM)$ and $m\in\MM$. The action $(S_g)_{g\in G}$ on $\Lip_0(\MM)$ is separately $w^*$-continuous.

\subsection{The space $\MM/G$}

In this subsection we work with a general Hausdorff topological group $G$. Suppose that $G$ acts by isometries on a metric space $\MM$. On the space of orbits consider the pseudometric $D(Gx,Gy):=\inf_{x'\in Gx,y'\in Gy} d(x',y')$. We consider the metric quotient, that is, points of $\MM/G$ are sets $[Gx]:=\{z\in\MM\colon D(Gx,Gz)=0\}=\overline{Gx}$ and the distance is given as \[d_{\MM/G}([Gx],[Gy]):=\inf_{x'\in [Gx],y'\in [Gy]} d(x',y').\] It is straightforward to verify that this coincides with the Hausdorff distance between the closed sets $[Gx]$ and $[Gy]$.

As the distinguished point in $\MM/G$, we choose [$G0]$, the closure of the orbit of $0$.

Given a Banach space $X$, by $\Lip^G_0(\MM,X)\subset \Lip_0(\MM,X)$ we denote the subset of functions which are $G$-invariant, that is, $f(x) = f(gx)$ for every $g\in G$ and $x\in \MM$. If $X=\Rea$, we write simply $\Lip_0^G(\MM)$ instead of $\Lip^G_0(\MM,\Rea)$.

\begin{Lemma}\label{lem:G-invLip}
Let $G$ be a group acting on $(\MM,d)$ by isometries and let $X$ be a Banach space. Then the mapping $\Psi: \Lip_0(\MM/G,X)\to \Lip^G_0(\MM,X)$ defined as \[\Psi(f)(x):=f([Gx]),\qquad f\in \Lip_0(\MM/G,X),\; x\in \MM\]
is a linear surjective isometry.
\end{Lemma}
\begin{proof}
For every $f\in \Lip_0(\MM/G,X)$ and two distinct points $x,y\in\MM$ we have 
\[
\|f([Gx])-f([Gy])\| \leq \Lip(f)d_{\MM/G}([Gx],[Gy])\leq \Lip(f) d(x,y),
\]
so it is easy to see that the mapping $\Psi$ is a norm-one linear operator.

Pick $f\in \Lip^G_0(\MM,X)$. Consider the mapping $Sf :\MM/G\to X$ given as $S f([Gx]):=f(x)$. Then given two distinct points $[Gx],[Gy]\in\MM$, for every $\varepsilon>0$ we pick $g,h\in G$ such that $d_{\MM/G}([Gx],[Gy]) > (1+\varepsilon)^{-1}d(gx,hy)$ and we obtain
\[\begin{split}
     \|Sf([Gx]])-Sf([Gy]])\| & = \|f(x)-f(y)\|\leq \Lip(f)d_{\MM}(gx,hy)\\ & \leq (1+\varepsilon)d_{\MM/G}([Gx],[Gy])\Lip(f),
\end{split}\]
 so $Sf$ is well defined and $S:\Lip^G_0(\MM,X)\to \Lip_0(\MM/G,X)$ given by $S(f):=Sf$ is a norm-one linear operator.
 
Finally, it is easy to check that $\Psi\circ S = \Id$ and $S \circ \Psi = \Id$, so $\Psi$ is a linear isometry with $\Psi^{-1} = S$.
\end{proof}

We conclude this subsection by isometric characterization of the space $\F(\MM/G)$.

\begin{Lemma}\label{lem:identifyMetricQuotient}
Let $\MM$ be a pointed metric space and let $G$ be a group acting on $\MM$ by isometries. Put $Y_G:=\cospn \{\delta(gx)-\delta(x)\colon g\in G, x\in\MM\}$. Then there is surjective linear isometry $\Lambda:\F(\MM)/Y_G\to\F(\MM/G)$ satisfying
\[\Lambda\Big([\delta(x)]\Big):=\delta([Gx]),\qquad x\in\MM.\]
Moreover, $Y_G = \Lip_0^G(\MM)_\perp$ and if $\Psi$ is the mapping from Lemma~\ref{lem:G-invLip} (where $X=\Rea$) then $\Lambda^*=\Psi$.
\end{Lemma}
\begin{proof}Let $\Psi$ be the mapping from Lemma~\ref{lem:G-invLip} for $X=\Rea$. Using the Banach-Diedonn\'e Theorem, see e.g. \cite{FHHMZ}*{Corollary 3.94}, and the fact that on bounded subsets of $\Lip_0(\MM)$ the $w^*$-convergence coincides with the pointwies convergence, it is easy to see that $\Psi$ is $w^*$-$w^*$ homeomorphism as a mapping from $\Lip_0(\MM/G)$ into $\Lip_0(\MM)$ and since by the definition we have
\[\Lip_0^G(\MM) = \{\delta(gx)-\delta(x)\colon g\in G, x\in\MM\}^{\perp},\]
its image is $w^*$-closed and moreover by the bipolar theorem we obtain $\Lip_0^G(\MM)_\perp = Y_G$.

It is well-known that whenever $M\subset X^*$ is $w^*$-closed subspace, then the mapping $I:M\to (X/_{M_\perp})^*$ defined by $I(f)([x]):=f(x)$, $f\in M$, $x\in X$ is $w^*$-$w^*$ homeomorphism and surjective linear isometry.

Thus, the mapping \[\Theta:\Lip_0(\MM/G)\to \Big(\F(\MM)/Y_G\Big)^*\] defined by
    \[\Theta(f)\big([\delta(x)]\big) = f\Big(\delta\big([Gx]\big)\Big),\qquad x\in\MM\]
    is surjective linear isometry and $w^*$-$w^*$ homeomorphism. Thus, the mapping $\Theta^*|_{\F(\MM)/Y_G}$ is isometry between $\F(\MM)/Y_G$ and $\F(\MM/G)$. Finally, it is easy to observe that we have $\Lambda = \Theta^*|_{\F(\MM)/Y_G}$ and $\Psi = \Lambda^*$.
\end{proof}

\subsection{Projections induced by group actions with bounded orbits}\label{subsec:mainSubsection} We start with a preliminary observation.

\begin{Lemma}\label{lem:maybeRedundant}
Let $\MM$ be a metric space, $G$ be an amenable group acting by affine isometries on $\MM$ which is either locally compact or SIN. If $M$ is right-invariant mean on $G$ and $R_G$, $I_G$ are as in Lemma~\ref{lem:amenableOperator} (applied to the action $\A_G$), then
    \[\forall x\in\MM\; \forall f\in\Lip_0^G(\MM):\qquad R_G(\delta(x))(f) = f(x).\]
\end{Lemma}
\begin{proof}
Pick $x\in\MM$ and $f\in\Lip_0^G(\MM)$. For every $g\in G$ we have \[f\big(T_g\delta(x)\big) = f(gx)-f(g0) = f(x),\]
which implies that
\[R_G(\delta(x))(f) = M\big(g\mapsto f\big(T_g\delta(x)\big)\big) = M\big(g\mapsto f(x)\big) = f(x).\]
\end{proof}

The following identifies ranges of mappings $R_G$ and $\widetilde{P_G}$ defined in Lemma~\ref{lem:amenableOperator} in the case when the group $G$ has bounded orbits and $X = \F(\MM)$.
\begin{Theorem}\label{thm:freeSpaceAmenableOperator}
Let $G$ be an amenable group acting continuously on a metric space $\MM$ by isometries with bounded orbits. Set $X=\F(\MM)$. Assume that $G$ is locally compact or SIN.

If $R_G$, $\widetilde{P_G}$ are the mappings from Lemma~\ref{lem:amenableOperator} (applied to the action $\A_G$), then $R_G[X]\subset X^{**}$ is a closed subspace linearly isometric to $\F(\MM/G)$,  $\widetilde{P_G}[X^*] = \Lip_0^G(\MM)$ and $\ker R_G = \cospn\{\delta(gx)-\delta(x)\colon x\in\MM,g\in G\}$.

In particular the following holds.
\begin{enumerate}
    \item $\Lip_0(\MM/G)$ is isometric to a $1$-complemented subspace of $\Lip_0(\MM)$.
    \item\label{it:freeWorse} If $\Lip_0^G(\MM)\cap (I_G)^\perp = \{0\}$, then $I_G$ is $1$-complemented in $\F(\MM)$.
\end{enumerate}
\end{Theorem}
\begin{proof}
We define a map $T':(\MM/G,d_{\MM/G})\rightarrow R_G[X]$ as follows. For $[Gx]\in\MM/G$ we set \[T'([Gx]):=R_G(\delta(x)).\]
We check that $T'$ is a well-defined $1$-Lipschitz map. Notice that by the ``Moreover'' part of Lemma~\ref{lem:amenableOperator} we have $R_G(\delta(x))=R_G(A_g\delta(x)) = R_G(\delta(gx))$ for all $g\in G$ and $x\in\MM$, so $R_G(\delta(x')) = R_G(\delta(x))$ for every $x'\in Gx$. 
Now, consider two points $[Gx], [Gy]\in \MM/G$, some $\varepsilon>0$, and pick $x'\in [Gx]$, $y'\in [Gy]$ such that $d(x',y')<d_{\MM/G}([Gx],[Gy])+\varepsilon$. Also, choose $x''\in Gx$, resp. $y''\in Gy$ satisfying $d(x'',x')<\varepsilon$, resp. $d(y'',y')<\varepsilon$. Thus we have 
\[\begin{split}
    \|T'([Gx])-& T'([Gy])\| =\|R_G(\delta(x''))-R_G(\delta(y''))\| \\ & = \|R_G(\delta(x''))-R_G(\delta(x')) + R_G(\delta(x')) - R_G(\delta(y')) + \\ & \qquad  R_G(\delta(y'))-R_G(\delta(y''))\|\\
& \leq d(x'',x') + d(x',y')+d(y',y'') \leq d_{\MM/G}([Gx],[Gy])+3\varepsilon.
\end{split}\]
Since $\varepsilon>0$ was arbitrary, we get that $T'$ is well-defined and $1$-Lipschitz. It follows that there is a linear extension $T:\F(\MM/G)\rightarrow X^{**}$ which is a norm-one linear surjection onto $\overline{R_G[X]}$. 

Now, we \emph{claim} that $R_G$ has closed range. Let us denote by $A':\MM\to \MM/G$ the surjection given by $A'(x) = [Gx]$, $x\in\MM$. Since $A'$ is $1$-Lipschitz, it extends to a linear surjection $A:\F(\MM)\to \F(\MM/G)$ and then we obviously have $R_G = T\circ A$, so $R_G[X] = T[\F(\MM/G)] = \overline{R_G[X]}$ which proves the claim.

In order to prove that $T$ is actually an isometry, pick some $z=\sum_{i=1}^n \alpha_i \delta([Gx_i])\in\F(\MM/G)$ and some $1$-Lipschitz $f\in \Lip_0(\MM/G)$ such that $\|z\|=|\sum_{i=1}^n \alpha_i f([Gx_i])|$. By Lemma~\ref{lem:G-invLip}, the mapping $\Psi(f)\in\Lip^G_0(\MM)\subseteq \Lip_0(\MM) = X^*$ is $1$-Lipschitz and therefore, using Lemma~\ref{lem:maybeRedundant} we obtain
\[\begin{split}\|T(z)\| & \geq \Big|\sum_{i=1}^n \alpha_i \big(R_G(\delta(x_i))\big)(\Psi(f))\Big|\\ & =\Big|\sum_{i=1}^n \alpha_i \Psi(f)(x_i)\Big| = \Big|\sum_{i=1}^n \alpha_i f([Gx_i])\Big|=\|z\|,
\end{split}\]
so the restriction of $T$ to a dense subset of $\F(\MM/G)$ is an isometry which implies that $T$ is isometry.

Finally, by Lemma~\ref{lem:amenableOperator} we have that $\widetilde{P_G}$ is a norm-one projection onto the space \[\widetilde{I_G} = \{f\in\Lip_0(\MM)\colon f(g^{-1}m)-f(g^{-1}0) = f(m) \text{ for every }m\in \MM\}.\]
Obviously, we have $\Lip_0^G(\MM)\subset I_G$. Conversely, by the ``Moreover'' part of Lemma~\ref{lem:amenableOperator} we have $\widetilde{P_G}(f)(m) = \widetilde{P_G}(f)(gm)$ for every $f\in\Lip_0(\MM)$, $m\in \MM$ and $g\in G$, so we obtain $\widetilde{I_G}\subset \Lip_0^G(\MM)$. Thus, $\widetilde{P_G}$ is a projection onto $\Lip^G_0(\MM)$. Using Lemma~\ref{lem:amenableOperator} \eqref{it:kernel}, we obtain
\[\begin{split}\ker R_G & = \Lip^G_0(\MM)_\perp = (\{\delta(gx)-\delta(x)\colon x\in\MM,g\in G\}^\perp)_\perp\\ & = \cospn \{\delta(gx)-\delta(x)\colon x\in\MM,g\in G\}.\end{split}\]

Finally, for the ``In particular'' part we use Lemma~\ref{lem:G-invLip}, Lemma~\ref{lem:identifyMetricQuotient} and the fact that \eqref{it:freeWorse} implies that \[\begin{split}(\ker R_G)^\perp\cap (I_G)^\perp & = \{\delta(gx)-\delta(x)\colon x\in\MM,g\in G\}^\perp \cap (I_G)^\perp\\ & = \Lip_0^G(\MM) \cap (I_G)^\perp = \{0\},\end{split}\]
so $\ker R_G + I_G = \F(\MM)$ and we may use Lemma~\ref{lem:amenableOperator} \eqref{it:projection}.
\end{proof}

As an immediate corollary we obtain an optimal result for compact groups.

\begin{Corollary}\label{cor:mainCpctGroup}
Let $G$ be a compact group acting continuously by isometries on a metric space $\MM$. Then $R_G[\F(\MM)]=I_G$ and it is linearly isometric to $\F(\MM/G)$.

In particular, $\F(\MM/G)$ is isometric to a $1$-complemented subspace of $\F(\MM)$.
\end{Corollary}
\begin{proof}
This follows from Theorem~\ref{thm:freeSpaceAmenableOperator} and from (b) in Lemma~\ref{lem:amenableOperator} by which $I_G = R_G[X]$ and $R_G$ is a projection.
\end{proof}

Let us now consider the situation when the group $G$ is not compact but at least we know $\F(M)$ is complemented in its bidual via a projection $Q$. In this case, by Theorem~\ref{thm:amenableLinearOperator} there is a natural condition on the projection $Q$ (implied by another conditions, see Lemma~\ref{lem:dualSpacesWorse} and Lemma~\ref{lem:lembed}) which implies that $I_G$ is complemented in $R_G[X]$ which in turn, by Theorem~\ref{thm:freeSpaceAmenableOperator}, is isometric to $\F(\MM/G)$. Hence, there are quite many situations when an application of our results gives that $I_G$ is isomorphic to a complemented subspace of $\F(\MM/G)$ for which by Lemma~\ref{lem:identifyMetricQuotient} we have $\F(\MM/G)\equiv \F(\MM)/Y_G$.

In what follows we find a sufficient condition under which there is actually an isomorphism between $I_G$ and $\F(\MM/G)$.

\begin{Theorem}\label{thm:complInBidualVariant}
Let $G$ be an amenable group acting continuously on a metric space $\MM$ by isometries with bounded orbits. Assume that $G$ is locally compact or SIN. Set $X=\F(\MM)$ and assume that $X$ is complemented in its bidual via a projection $Q$ satisfying $\T_G(G)\cdot \ker Q\subset \ker Q$, that is, for every $g\in G$ we have $T_g(\ker Q)\subset \ker Q$. Let $P_G:X\to X$ be the projection onto $I_G$ from the statement of Theorem~\ref{thm:amenableLinearOperator} (applied to the action $\A_G$) and $\Psi:\Lip_0(\MM/G)\to \Lip_0^G(\MM)$ the mapping from Lemma~\ref{lem:G-invLip}.

If $Y\subset \Lip_0(\MM/G)$ separates the points of $\F(\MM/G)$ and $\ker P_G\subset \Psi(Y)_\perp$, then $\ker P_G = \Lip_0^G(\MM)_\perp$ and $P_G[X]$ is $\|Q\|$-isomorphic to $\F(\MM/G)$.

In particular, we have \[\F(\MM/G)\simeq P_G[X]\stackrel{C}{\hookrightarrow}\F(\MM).\]
\end{Theorem}
\begin{proof}We define a map $T':(\MM/G,d_{\MM/G})\rightarrow \F(\MM)$ as follows. For $[Gx]\in\MM/G$ we set \[T'([Gx]):=P_G(\delta(x)).\]
Similarly as in the proof of Theorem~\ref{thm:freeSpaceAmenableOperator}, using that by the ``Moreover'' part of Theorem~\ref{thm:amenableLinearOperator} we have $P_G(\delta(x))=P_G(A_g\delta(x)) = P_G(\delta(gx))$ for all $g\in G$ and $x\in\MM$, we check that $T'$ is a well-defined $1$-Lipschitz map.

Thus, there is a linear extension $T:\F(\MM/G)\rightarrow \F(\MM)$ which is a norm-one linear surjection onto $P_G[X] = I_G$ and we have $T(\delta([Gx])) = P_G(\delta(x))$ for every $x\in \MM$.

Similarly, there is a linear mapping $S:\F(\MM)\to\F(\MM/G)$ with $\|S\|\leq 1$ and $S(\delta(x)) = \delta([Gx])$ for every $x\in \MM$. Note that $S^* = \Psi$.

Let us observe that we have $S\circ T = \Id$. Pick $x\in \MM$ and $f\in Y$. Then we have
\[\begin{split}
    f\Big((S\circ T - \Id)\big(\delta([Gx])\big)\Big) & = f\Big(S(P_G(\delta(x)))\Big) - S^* f(\delta(x))\\ & = S^* f\big(P_G(\delta(x)) - \delta(x)\big) = 0,
\end{split}\]
where in the last equality we used that $P_G(\delta(x)) - \delta(x)\in \ker P_G\subset \Psi(Y)_\perp$ and $S^* f\in \Psi(Y)$. Therefore, since $Y$ separates the points of $\F(\MM/G)$ and $\delta(\MM/G)\subset \F(\MM/G)$ is linearly dense, we have that $S\circ T-\Id = 0$, so $S\circ T = \Id$ as required.

Thus, $TS = P_G$ is the projection onto a space $\|Q\|$-isomorphic to $\F(\MM/G)$.

By Theorem~\ref{thm:amenableLinearOperator} and Theorem~\ref{thm:freeSpaceAmenableOperator} we have $\Lip_0^G(\MM)_\perp = \ker R_G\subset \ker P_G$ and since $\Psi(Y)$ is $w^*$-dense in $\Lip_0^G(\MM)$, we have $\Psi(Y)_\perp = \Lip_0^G(\MM)_\perp$ and $\ker P_G \subset \Lip_0^G(\MM)_\perp$ follows from the assumption. Hence, we obtain $\ker P_G = \Lip_0^G(\MM)_\perp$.
\end{proof}

\subsection{Projections in Lipchitz-free spaces which are dual spaces}\label{subsec:mainSubsectionDual}

Let us start with an abstract result following easily from our previous considerations.

\begin{Corollary}\label{cor:amenableProjectionDualSpace}
Let $G$ be an amenable group acting continuously on a metric space $\MM$ by isometries with bounded orbits. Let $X=\F(\MM)$ be a dual space with $X\equiv Y^*$ for some $Y\subset \Lip_0(\MM)$ and assume that we have $S(G)\cdot Y\subset Y$, that is, for every $g\in G$ and $f\in Y$ we have $S_gf\in Y$. Assume that $G$ is locally compact or SIN. Let $P_G:X\to X$ be the projection from Theorem~\ref{thm:amenableLinearOperator} (applied to the action $\A_G$) and $\Psi:\Lip_0(\MM/G)\to \Lip_0^G(\MM)$ the mapping from Lemma~\ref{lem:G-invLip}. Put $Z:=\Psi^{-1}(Y\cap \Lip_0^G(\MM))\subset \F(\MM/G)^*$.

If $Z$ separates the points of $\F(\MM/G)$, then $\F(\MM/G)$ is isometric to a $1$-complemented subspace of $\F(\MM)$.
\end{Corollary}
\begin{proof}
We start with the following claim.
\begin{equation}\label{claim:invariant2}\forall x\in\MM\;\forall \Phi\in\Lip_0^G(\MM)\cap Y:\qquad P_G(\delta(x))(\Phi) = \Phi(x).\end{equation}
Indeed, this follows directly from Lemma~\ref{lem:maybeRedundant} and the fact that by Lemma~\ref{lem:dualSpacesWorse} we have $P_G(m) = R_G(m)|_{Y}$ for every $m\in\F(\MM)$.

Thus, we obtain $\ker P_G\subset \Psi(Z)_\perp$. Thus, using Lemma~\ref{lem:dualSpacesWorse} we observe that assumptions of Theorem~\ref{thm:complInBidualVariant} are satisfied (applied to the set $Z$ separating the points of $\F(\MM/G)$). Thus, an application of Theorem~\ref{thm:complInBidualVariant} finishes the proof.
\end{proof}

The case when $\F(\MM)$ is a dual space has recently been thoroughly investigated in \cite{AlNew}. Let us recall few notions.

Let $\MM$ be a metric space. A Lipschitz function $f:\MM\to \Rea$ is said to be \emph{locally flat} if for every $\varepsilon > 0$ there exists $\delta > 0$ such that $|f(x)-f(y)|\leq \varepsilon d(x,y)$ whenever $d(x,y)<\delta$. By $\lip_0(\MM)$ we denote  the space of all functions in $\Lip_0(\MM)$ that
are locally flat and moreover \emph{flat at infinity}, i.e. such that $\|f|_{\MM\setminus B(0,r)}\|_{\Lip}\to 0$ as $r\to\infty$. We say $\MM$ is \emph{proper} if every closed ball is compact. Finally, we say $\MM$ is \emph{purely $1$-unrectifiable} if, for every $A\subset \Rea$ and Lipschitz map $f:A\to \MM$, the $1$-dimenensional Hausdorff measure of $f(A)$ equals $0$. By \cite{Kirch}*{Theorem 9}, $\MM$ is purely $1$-unrectifiable if and only if it does not contain a bi-Lipschitz image of a set $K\subset \Rea$ of positive Lebesgue measure.

Now, we are ready to recall the following crucial result which was proved recently in \cite{AlNew} (see Theorem 3.1 and the comment in the first paragraph of Subsection 3.1 therein).

\begin{Theorem}\label{thm:al}
Let $\MM$ be a proper space. Then the following conditions are equivalent:
\begin{itemize}
    \item $\MM$ is purely $1$-unrectifiable;
    \item $\F(\MM)$ is a dual space;
    \item $L_1\not\hookrightarrow \F(\MM)$;
    \item $\F(\MM) = \lip_0(\MM)^*$ isometrically.
\end{itemize}
Moreover, if the conditions above hold, then $\F(\MM)$ is $L$-embedded.
\end{Theorem}

Combined with our results we obtain the following.

\begin{Theorem}\label{thm:dual}
Let $G$ be an amenable group acting continuously on a proper purely $1$-unrectifiable metric space $\MM$ by isometries with bounded orbits. Assume that $G$ is locally compact or SIN. Then the following conditions are equivalent.
\begin{enumerate}
    \item\label{it:dual2} $\MM/G$ is purely $1$-unrectifiable.
    \item\label{it:compl} $\F(\MM/G)$ is isometric to a $1$-complemented subspace of $\F(\MM)$.
    \item\label{it:embed} $\F(\MM/G)\hookrightarrow \F(\MM)$.
\end{enumerate}
\end{Theorem}
\begin{proof}Note that since the action of $G$ has bounded orbits, every closed ball in $\MM/G$ is contained in a Lipschitz image of a closed ball in $\MM$, so since $\MM$ is proper, $\MM/G$ is proper as well. Further, by Theorem~\ref{thm:al}, $\lip_0(\MM)$ separates the points of $\F(\MM)$.

If $\MM/G$ is purely $1$-unrectifiable, applying Theorem~\ref{thm:al} we observe that the space $\lip_0(\MM/G)$ separates the points of $\F(\MM/G)$ and that $\F(\MM/G)$ is $L$-embedded. Moreover, it is easy to observe that if $\Psi:\Lip_0(\MM/G)\to \Lip_0^G(\MM)$ is the mapping from Lemma~\ref{lem:G-invLip}, then \[\Psi^{-1}(\lip_0(\MM)\cap \Lip_0^G(\MM)) = \lip_0(\MM/G).\]
Thus, by Corollary~\ref{cor:amenableProjectionDualSpace} and Lemma~\ref{lem:lembed} we obtain that \eqref{it:compl} holds.

Trivially, \eqref{it:compl} implies \eqref{it:embed}. Finally, using Theorem~\ref{thm:al} and the fact that both $\MM$ and $\MM/G$ are proper, it is easy to see that \eqref{it:embed} implies \eqref{it:dual2}.
\end{proof}

We do not know whether the equivalent conditions in Theorem~\ref{thm:dual} are actually always satisfied, see Question~\ref{q:unrect}. On the other hand, note that Lipschitz image of a compact purely $1$-unrectifiable metric space need not be purely $1$-unrectifiable as witnessed e.g. by the mapping $\Id:([0,1],|\cdot|^{1/2})\to ([0,1],|\cdot|)$. Thus, the following seems to be an interesting consequence.

\begin{Corollary}\label{cor:unrectGroup}
Let $G$ be a compact group acting continuously on a proper purely $1$-unrectifiable metric space $\MM$ by isometries. Then $\MM/G$ is purely $1$-unrectifiable
\end{Corollary}
\begin{proof}
Follows directly from Theorem~\ref{thm:dual} and Corollary~\ref{cor:mainCpctGroup}.
\end{proof}

\section{Problems and applications}

One of the main advantages of the methods presented in this paper is that it allows to conclude that free spaces over the spaces $\MM/G$, in many cases, enjoy all the properties that the spaces $\F(\MM)$ do provided these are inherited to complemented subspaces.

In \cite{DP20}, the authors used this observation in order to prove that $\F(G)$ has MAP whenever $G$ is a compact group. More precisely, the projection considered in this paper enabled the authors to reduce the situation to the case when $G$ is a Lie group. Let us comment on this kind of reduction, which could have some further applications.

 Fix such a Banach space property $\PP$ that is inherited to complemented subspaces (resp. $1$-complemented subspaces) e.g. the bounded approximation property (resp. the metric approximation property).

To illustrate this idea more properly on examples, consider homogeneous metric spaces. We recall that a metric space $\MM$ is \emph{homogeneous} if each point of $\MM$ is `indistinguishable' from each other, i.e. there is no metric property that can tell two points from each other. Formally, this can be defined by saying that the isometry group of $\MM$ acts transitively on $\MM$, that is, for each pair $x,y\in\MM$ there is an isometry $\phi$ of $\MM$ satisfying $\phi(x)=y$.

Suppose now that $\MM$ is moreover proper, i.e. all the balls are compact. Then for its isometry group with the pointwise convergence topology, denoted by $G$, it is well known and not difficult to verify that
\begin{itemize}
\item $G$ is locally compact;
\item $G$ acts on $\MM$ transitively;
\item the stabilizer of any point $x\in\MM$ is a compact subgroup $K$ and $G/K$ is homeomorphic with $\MM$.
\end{itemize}
We briefly sketch the proof for the convenience of the reader. By definition of homogeneity, $G$ acts transitively. Fix a countable dense subset $(x_n)_{n\in\Nat}\subseteq\MM$, fix $g\in G$ and let us show that it has a compact neighborhood. Set $U=\{h\in G\colon d(h x_1, g x_1)\leq 1\}$, we claim it is compact. It suffices to check that every sequence $(g_n)_{n\in\Nat}\subseteq U$ has a convergent subsequence. Notice that for every $m\in\Nat$ the set $\bigcup_{i\in\Nat} \{g_i x_1,\ldots,g_i x_m\}$ is bounded (each $g_ix_m$ is contained in the ball $B(gx_1,1+d(x_m,x_1))$), therefore precompact. Thus, using a diagonal argument,  we can find a subsequence $(g_{i_n})_n$ such that for every $m\in\Nat$, $(g_{i_n} x_m)_n$ is convergent. It follows that the sequence $(g_{i_n})_n$ is convergent. This finishes the proof of the claim. This also implies that the stabilizer of any point is compact (a closed subset of a compact neighborhood). For the last assertion see \cite{gao}*{Theorem 3.2.4}.

If in addition there exists a left-invariant metric on $G$ such that the metric on $\MM=G/K$ is the quotient metric, then it follows from Corollary~\ref{cor:mainCpctGroup} (applied to the action of $K$ on the metric space $G$) that $\F(\MM)$ has $\PP$ if $\F(G)$ does.

\begin{Example}
Let $\MM$ be a symmetric Riemannian manifold equipped with a Riemannian distance. Let $G$ be the identity component of its isometry group equipped with a left-invariant Riemannian distance (since it is a Lie group). Then $\F(\MM)$ has $\PP$ if $\F(G)$ does. Indeed, see \cite{Mor}*{\S 1.2 \# 7} for the fact that $\MM$ is isometric to $G/K$, where $G$ is as above and $K\subset G$ a compact subgroup.
\end{Example}
This suggests it is important to understand the structure of free spaces over (connected) Lie groups, with their canonical left-invariant Riemannian distance, but also with other compatible left-invariant metrics. Indeed, by the celebrated Gleason-Yamabe's solution to Hilbert's fifth problem, for every connected locally compact group $G$ and an arbitrarily small neighborhood $U$ of the identity in $G$ there exists a compact subgroup $K\subseteq U$ such that $G/K$ is a connected Lie group, see \cite{Tao}*{Theorem 1.1.17}, which can be equipped with the quotient metric of the left-invariant metric on $G$. Therefore, if $\MM$ is a homogeneous connected proper metric space whose metric can be lifted to the left-invariant metric of its isometry group, then a lot of information about $\F(\MM)$ can be derived just from the information about the spaces $\F(G)$, where $G$ is a (connected) Lie group.

\begin{Suggestion}
Study Banach space properties that are inherited to (one-)complemented subspaces on free spaces over connected Lie groups with left-invariant metrics.
\end{Suggestion}

Quite many natural questions arise when wondering under which conditions we have $\F(\MM/G)\complemented \F(\MM)$. Let us conclude this paper by mentioning few of those.

From the point of view of Banach space theory, the following seems to be an interesting problem.
\begin{Question}Let $X$ be a Banach space and $Y\subset X$ its closed subspace. Is it true that $\Lip_0(X/Y)\complemented \Lip_0(X)$?
\end{Question}
Note that, since any separable Banach space is isomorphic to a quotient of $\ell_1$ a positive answer would imply that for any separable Banach space $X$ we have $\Lip_0(X)\complemented \Lip_0(\ell_1)$ and in particular, by the Pe\l cyz\'nski decomposition method and \cite{K15}*{Theorem 3.1}, we would obtain  $\Lip_0(X)\simeq \Lip_0(\ell_1)$ whenever $\ell_1\complemented X$. This seems to be quite a strong consequence. In particular, since it is known that $\ell_1\complemented U\simeq \F(U)$ for the Pe\l czy\'nski's universal basis space $U$, see \cite{GodefroyKalton2003}*{Remark on p.139}, we would have $U^*\simeq \Lip_0(\ell_1)$. Moreover, since by \cite{J72}*{Theorem 3} $U^*$ does not have AP, $\Lip_0(\ell_1)$ would not have AP. 

Another interesting question from the point of view of Banach space theory is the following. Recall that a net in a Banach space is a subset which is $a$-dense and $b$-separated for some $a,b >0$ and note that for any Banach space $X$ we may find a net $\mathcal{N}_X\subset X$ which is also additive subgroup, see \cite{DOSZ08}*{Theorem 5.5}, so the question seems to be naturally connected to the research handled in this paper. 
\begin{Question}\label{q:lipNetQuotient}Let $X$ be an infinite-dimensional Banach space. Does there exist a net $\mathcal{N}_X\subset X$ which is also additive subgroup such that  $\Lip_0(X/\mathcal{N}_X)\complemented \Lip_0(X)$?
\end{Question}
Note that it is even open whether for infinite-dimensional spaces we have $\Lip_0(\mathcal{N}_X)\complemented \Lip_0(X)$, see \cite{CCD19}*{Question 3 and a comment below}, so the above more-or-less heads towards a natural question whether we have $\Lip_0(X)\simeq \Lip_0(\mathcal{N}_X)\oplus \Lip_0(X/\mathcal{N}_X)$ isomorphically.

Let us also note that if an additive subgroup $\mathcal{N}_X\subset X$ is $K$-dense in $X$ and we equip $X$ with the metric $d(x,y):=\min\{\|x-y\|,2K\}$, then $X/\mathcal{N}_X$ is isometric to $(X,d)/\mathcal{N}_X$, so using Theorem~\ref{thm:freeSpaceAmenableOperator} we obtain $\Lip_0((X,d)/\mathcal{N}_X)\complemented \Lip_0(X,d)$. Thus, a possible way of providing a positive answer to Question~\ref{q:lipNetQuotient} would be to show that $\Lip_0(X,d)\complemented \Lip_0(X,\|\cdot\|)$.

In our proofs we were using the assumption that $G$ has bounded orbits. In this case one could ask e.g. the following which aims at pushing forward what was initiated in Theorem~\ref{thm:freeSpaceAmenableOperator}.
\begin{Question}
Let $G$ be a metric group which is amenable and locally compact or SIN. Let $H\subset G$ be its bounded subgroup. Is it true that $\F(G/H)\complemented \F(G)$?
\end{Question}

Note that using Theorem~\ref{thm:freeSpaceAmenableOperator} it would be sufficient to prove that $I_G$ separates the points of $\Lip_0^G(\MM)$ and moreover $I_G\simeq R_G[X]$ isomorphically (where $\Lip_0^G(\MM)$, $I_G$ and $R_G$ are as in Theorem~\ref{thm:freeSpaceAmenableOperator}).

Finally, let us note that, as mentioned above, we do not know an answer to the following.
\begin{Question}\label{q:unrect}
Let $G$ be an amenable group acting on a proper purely $1$-unrectifiable metric space $\MM$ by isometries with bounded orbits. Assume that $G$ is locally compact or SIN. Is then $\MM/G$ purely $1$-unrectifiable?
\end{Question}
By Corollary~\ref{cor:unrectGroup}, the answer is positive for compact groups.


\end{document}